\documentclass[12pt, reqno]{amsart}
\usepackage{amsmath, amsthm, amscd, amsfonts, amssymb, graphicx, color}
\usepackage[bookmarksnumbered, colorlinks, plainpages]{hyperref}
\hypersetup{colorlinks=true,linkcolor=red, anchorcolor=green, citecolor=cyan, urlcolor=red, filecolor=magenta, pdftoolbar=true}

\newtheorem{theorem}{Theorem}[section]

\newtheorem*{theorem*}{Question}
\newtheorem{proposition}[theorem]{Proposition}
\newtheorem{corollary}[theorem]{Corollary}
\theoremstyle{definition}

\newtheorem{example}[theorem]{Example}

\theoremstyle{remark}
\newtheorem{remark}[theorem]{Remark}
\numberwithin{equation}{section}

\begin{document}

\setcounter{page}{1}

\title[Convexity of the Berezin range of finite rank operators]{Convexity of the Berezin range of finite rank operators}
\author[Athul Augustine, M. Garayev \MakeLowercase{and} P. Shankar]{Athul Augustine, M. Garayev \MakeLowercase{and} P. Shankar}

\address{Athul Augustine, Department of Mathematics, Cochin University of Science And Technology,  Ernakulam, Kerala - 682022, India. }
\email{\textcolor[rgb]{0.00,0.00,0.84}{athulaugus@gmail.com, athulaugus@cusat.ac.in}}

\address{M. Garayev, Department of Mathematics, College of Science , King Saud University, P.OBox 2455Riyadh 11451, Saudi Arabia }
\email{\textcolor[rgb]{0.00,0.00,0.84}{mgarayev@ksu.edu.sa}}

\address{P. Shankar, Department of Mathematics, Cochin University of Science And Technology,  Ernakulam, Kerala - 682022, India.}
\email{\textcolor[rgb]{0.00,0.00,0.84}{shankarsupy@gmail.com, shankarsupy@cusat.ac.in}}




\begin{abstract}
For a bounded linear operator $T$ acting on a reproducing kernel Hilbert space $\mathcal{H}(\Omega)$ over a nonempty set $\Omega$, the Berezin range of $T$ is defined by
\[
\mathrm{Ber}(T)=\left\{\langle T\hat{k}_{\lambda},\hat{k}_{\lambda}\rangle_{\mathcal{H}} : \lambda \in \Omega \right\}
\]
and the Berezin radius is given by
\[
\mathrm{ber}(T)=\sup\left\{ |\gamma| : \gamma \in \mathrm{Ber}(T) \right\},
\]
where $\hat{k}_{\lambda}$ denotes the normalized reproducing kernel at $\lambda \in \Omega$. In this paper, we study the convexity of the Berezin range of finite rank operators on the Hardy space and the Bergman space over the unit disc $\mathbb{D}$. We present applications of some scalar inequalities to get some operator inequalities. A characterization of closure of the numerical range of reproducing kernel Hilbert space operator in terms of convex hull of its Berezin range is also discussed.
\end{abstract}
\maketitle

\section{Introduction}

A reproducing kernel Hilbert space (RKHS) is a Hilbert space $\mathcal{H}(\Omega)$ of complex-valued functions on some set $\Omega$ such that the point evaluation at $\lambda\in\Omega$, $E_\lambda(f) = f(\lambda)$ is a bounded linear functional on $\mathcal{H}(\Omega)$. Then, by the Riesz representation theorem, for each $ \lambda \in \Omega $, there exists a unique element $k_{\lambda} \in \mathcal{H}(\Omega) $ such that $E_\lambda(f) =f(\lambda) = \langle f, k_{\lambda} \rangle$ for all $f \in \mathcal{H}(\Omega).$ The collection $ \{ k_{\lambda} : \lambda \in \Omega \}$ is called  the reproducing kernels of the space $ \mathcal{H}(\Omega) $. 

Let $\{e_n\}$ be an orthonormal basis for an RKHS. Then, the reproducing kernel of this space is given by $k_\lambda(z) = \sum_{n} \overline{e_n(\lambda)}e_n(z)$. Well-known examples of RKHSs are the Hardy space, the Bergman space, the Dirichlet space and the Fock space. For a comprehensive discussion on the general theory of RKHS, we refer the reader to \cite{paulsen2016introduction}.

For a bounded linear operator $T$ acting on a reproducing kernel Hilbert space $\mathcal{H}(\Omega)$, the Berezin transform of $T$ at $\lambda \in \Omega$ is $\widetilde{T}(\lambda):= \langle T \hat{k}_\lambda, \hat{k}_\lambda\rangle_{\mathcal{H}(\Omega)}$, where $\hat{k}_\lambda =\frac{k_\lambda}{\|k_\lambda\|}$ is the normalized reproducing kernel at $\lambda$. The Berezin transform was first introduced by F. A. Berezin in \cite{berezin1972covariant}. The Berezin range and the Berezin radius of an operator $T$ are defined, respectively as
\[
\mathrm{Ber}(T):=\left\{\langle T\hat{k}_{\lambda},\hat{k}_{\lambda}\rangle_{\mathcal{H}(\Omega)} : \lambda \in \Omega \right\}
\]
and
\[
\mathrm{ber}(T):=\sup\left\{ |\gamma| : \gamma \in \mathrm{Ber}(T) \right\}.
\]
Recall that, for a bounded linear operator $T$ on a Hilbert space $\mathcal{H}$, the numerical range and numerical radius of $T$ are defined, respectively as
$$ W(T) := \{\langle Tu,u\rangle :u\in\mathcal{H}~~ \text{and}~~ \|u\|=1\}$$
and 
$$w(T) := \sup\{|\langle Tu,u\rangle |:u\in\mathcal{H}~~ \text{and}~~ \|u\|=1 \}.$$
It is clear from the definition of the Berezin range that the Berezin range is contained in the numerical range, i.e., $ \text{Ber}(T) \subseteq W(T)$, which implies that $ \text{ber}(T)\leq w(T)$. It is also obvious that the Berezin radius of an operator $T$ does not exceed its norm, i.e., $ \text{ber}(T)\leq \|T\|$. 

Over the years, the geometry of the numerical range and many numerical radius inequalities have been studied extensively by many mathematicians. So, many questions, which are well-studied for the numerical range and the numerical radius of an operator $T$, can be naturally asked for the Berezin range and the Berezin radius of $T$. The convexity of the Berezin range has been investigated in several works, including \cite{augustine2023composition, augus2024, cowen22, karaev2013reproducing}. Studies on Berezin radius inequalities can be found in \cite{5,7,10,14,8,19,18,17}.

 In this paper, we discuss the convexity of the Berezin range of finite rank operators on the Hardy and the Bergman spaces over the unit disc $\mathbb{D}$. We prove some inequalities involving the Berezin norm and the Berezin radius of an operator $T\in B(\mathcal{H}(\Omega)).$ 

This article is demarcated into five sections, beginning with the introductory section. In section 2, we characterize the convexity of the Berezin range of some finite rank operators on the Hardy space, $H^2(\mathbb{D})$. In section 3, we characterize the convexity of the Berezin range of some finite rank operators on the Bergman space, $A^2(\mathbb{D})$. In Section 4, we give applications of some scalar inequalities. This applications will treat inequalities , where Berezin radius, Berezin norm and operator means are discussed. In Section 5, we discuss the relationship between the closure of the numerical range and the convex hull of the Berezin range of operators.

\section{Finite rank operators on Hardy space}
Consider the open unit disc $\mathbb{D}$ in the complex plane $\mathbb{C}$ and let $\text{Hol}(\mathbb{D})$ be the collection of holomorphic functions on $\mathbb{D}$. Then the classical Hilbert Hardy space on the unit disc, $H^2(\mathbb{D})$, is
$$H^2(\mathbb{D}):= \left\{ f(z) = \sum_{n\geq 0} a_nz^n \in\text{Hol}(\mathbb{D}) : \sum_{n\geq 0} |a_n|^2 < \infty\right\}.$$
For $f,g\in H^2(\mathbb{D})$ and $f(z) = \sum_{n\geq 0} a_nz^n$ and $g(z) = \sum_{n\geq 0} b_nz^n$, the inner product is given by $$\langle f,g\rangle_{H^2(\mathbb{D}} = \sum_{n\geq 0} a_n\overline{b_n}.$$
$H^2(\mathbb{D})$ is a reproducing kernel Hilbert space with $\|f\|_{H^2(\mathbb{D})} = \sum_{n\geq 0} |a_n|^2$. For $z,\lambda \in\mathbb{D}$, the reproducing kernel at $\lambda$ for $H^2(\mathbb{D})$ is
$$ k_{\lambda}(z) = \frac{1}{1-\bar{\lambda}z}.$$
This function is known as the Szeg\"{o} kernel.

Karaev in \cite{karaev2013reproducing} calculated the Berezin range of the rank one operator, $A(f) = \langle f,z \rangle z$ acting on $H^2(\mathbb{D})$. For $\lambda \in \mathbb{D}$, we have 
$$ A(k_{\lambda}) = \langle k_{\lambda},z \rangle z = \overline{\lambda}z.$$
The Berezin transform at $\lambda$ is
$$\widetilde{A}(\lambda) =\langle A\hat{k}_{\lambda},\hat{k}_{\lambda}\rangle=\frac{1}{||k_z||^2}\langle A{k_{\lambda}},{k_{\lambda}}\rangle=(1-|\lambda|^2) |\lambda|^2.$$
Thus $\text{Ber}(A) = [0,\frac{1}{4}]$ which is a convex set. 

Now, we prove the convexity of the Berezin range of the rank one operator, $A(f) = \langle f,z^n \rangle z^n$, for $n \in \mathbb{N}$.
\begin{proposition}\label{2.1}
Let $A(f) = \langle f,z^n \rangle z^n$ be a rank one operator on $H^2(\mathbb{D})$. Then the Berezin range of $A$ is $\text{Ber}(A)= \left[0,\left(\frac{1}{n+1}\right)\left(\frac{n}{n+1}\right)^n\right]$, which is convex in $\mathbb{C}$.
\end{proposition}
\begin{proof}
For $\lambda \in \mathbb{D}$, $ A(k_{\lambda}) = \langle k_{\lambda},z^n \rangle z^n = \overline{\lambda}^nz^n.$
The Berezin transform at $\lambda$ is
$$\widetilde{A}(\lambda)=(1-|\lambda|^2)\langle \overline{\lambda}^nz^n,{k_{\lambda}}\rangle=(1-|\lambda|^2) |\lambda|^{2n}=|\lambda|^{2n}- |\lambda|^{2n+2},$$
 where $|\lambda|\in [0,1)$. Observe that  $\widetilde{A}(\lambda) = \widetilde{A}(|\lambda|)$ is a real function. Now we differentiate $\widetilde{A}(\lambda)$ with respect to $|\lambda|$ and equate it to zero to find the extreme points.
$$0 =2n|\lambda|^{2n-1} - (2n+2)|\lambda|^{2n+1}=2|\lambda|^{2n-1}(n-(n+1)|\lambda|^{2}).$$
  This happens if and only if $|\lambda|=0$ or $|\lambda|^2=\frac{n}{n+1}$. Now if $|\lambda|=0$ then $\widetilde{A}(\lambda) = 0$. If $|\lambda|^2=\frac{n}{n+1}$, then
 $$\widetilde{A}(\lambda) = \left(\frac{n}{n+1}\right)^n- \left(\frac{n}{n+1}\right)^{n+1}=\left(\frac{1}{n+1}\right)\left(\frac{n}{n+1}\right)^n.$$
  Thus the Berezin range of $A$, $\text{Ber}(A)= \left[0,\left(\frac{1}{n+1}\right)\left(\frac{n}{n+1}\right)^n\right]$, which is a convex set in $\mathbb{C}$.
\end{proof}
  
  \begin{theorem}\label{2.2}
  Let $ A_n(f) = \sum_{i=1}^{n}\langle f,a_iz^i \rangle a_iz^i$ be a finite dimensional operator on $H^2(\mathbb{D})$, where $a_i \in \mathbb{C}$. Then the Berezin range of $A_n$ is convex in $\mathbb{C}$.
  \end{theorem}
  \begin{proof}
   For $\lambda \in \mathbb{D}$, we have 
$A_n(k_{\lambda}) = \sum_{i=1}^{n} |a_i|^2\overline{\lambda}^iz^i.$
The Berezin transform of $A_n$ at $\lambda$ is
  	\begin{equation*}
  		\begin{split}
  			\widetilde{A}_n(\lambda) &=(1-|\lambda|^2) \sum_{i=1}^{n} |a_i|^2|\lambda|^{2i}\\
  			&=|a_1|^2|\lambda|^2 + (|a_2|^2 - |a_1|^2)|\lambda|^4+ \cdots + (|a_{n}|^2 - |a_{n-1}|^2)|\lambda|^{2n} - |a_n|^2|\lambda|^{2n+2}.
  		\end{split}
  	\end{equation*}
  Thus, the Berezin transform is a real continuous function in the complex plane $\mathbb{C}$. Therefore $\text{Ber}(A_n)$ is convex.
  \end{proof}
  \begin{corollary}\label{2.3}
 Let $ A_n(f) = \sum_{i=1}^{n}\langle f,a_iz^i \rangle a_iz^i$ be a finite dimensional operator on $H^2(\mathbb{D})$, where $a_i \in \mathbb{C}$.
\begin{itemize}
\item[(i)] If $|a_i|=|a|~~ \forall~~ i\in \{1,2,\cdots,n\}$. Then the Berezin range of $A_n$, $\text{Ber}(A_n) = \left[0,|a|^2\sqrt[n]{\frac{1}{(n+1)}}\left(\frac{n}{n+1}\right)\right].$
\item[(ii)] If $a_i \in \mathbb{T},~~ \forall~~ i\in \{1,2,\cdots,n\}$, then $\text{Ber}(A_n) = \left[0,\sqrt[n]{\frac{1}{(n+1)}}\left(\frac{n}{n+1}\right)\right].$
\end{itemize} 
  \end{corollary}
  \begin{proof}
  From the proof of Theorem \ref{2.2}, we have
  $$\widetilde{A}_n(\lambda) = |a_1|^2|\lambda|^2 + (|a_2|^2 - |a_1|^2)|\lambda|^4+ \cdots + (|a_{n}|^2 - |a_{n-1}|^2)|\lambda|^{2n} - |a_n|^2|\lambda|^{2n+2}.$$
$(i)$ Since $|a_i|=|a|,~~ \forall~~ i\in \{1,2,\cdots,n\}$, we have
   $$\widetilde{A_n}(\lambda)= |a|^2|\lambda|^2 - |a|^2|\lambda|^{2n+2}.$$
  Equating the derivative to zero, we get
$$0 =2|a|^2|\lambda| - (2n+2)|a|^2|\lambda|^{2n+1}=2|\lambda|\left[|a|^2-(n+1)|a|^2|\lambda|^{2n}\right].$$
  This happens if and only if $|\lambda|^2 = \sqrt[n]{\frac{1}{(n+1)}}.$ Substituting these extreme points in $\widetilde{A}_n(\lambda)$, we get 
 $\text{Ber}(A_n) = \left[0,|a|^2\sqrt[n]{\frac{1}{(n+1)}}\left(\frac{n}{n+1}\right)\right].$
 
 $(ii)$ follows from $(i)$ by substituting $|a_i|=1$.
  \end{proof}
  
Consider a compact self adjoint operator of the form 
$$A(f) =   \sum_{n=1}^{\infty}\langle f,az^n \rangle az^n$$
where $a\in \mathbb{D}$, acting on $H^2(\mathbb{D})$. We calculate the Berezin range of this operator.
\begin{theorem}\label{2.4}
Let $A(f) =   \sum_{n=1}^{\infty}\langle f,az^n \rangle az^n$ be a compact self-adjoint operator on $H^2(\mathbb{D})$, where $a \in \mathbb{D}$.
Then the Berezin range of $A$, $\text{Ber}(A)=[0,|a|^2)$, which is convex in $\mathbb{C}$.
\end{theorem}
\begin{proof}
For $\lambda \in \mathbb{D}$, $A(k_{\lambda}) = |a|^2\sum_{n=1}^{\infty}\overline{\lambda}^n z^n.$
The Berezin transform of $A$ at $\lambda$ is
  	\begin{equation*}
  		\begin{split}
  			\widetilde{A}(\lambda) &=(1-|\lambda|^2)\left\langle |a|^2 \sum_{n=1}^{\infty} \overline{\lambda}^nz^n,{k_{\lambda}}\right\rangle\\
  			&=|a|^2(1-|\lambda|^2) \sum_{n=1}^{\infty} |\lambda|^{2n}\\
  			&=|a|^2|\lambda|^2.
  		\end{split}
  	\end{equation*}
  Since $\lambda \in \mathbb{D}$, we have $|\lambda| \in [0,1).$ Thus the Berezin range of $A$, $\text{Ber}(A)=[0,|a|^2)$, which is a convex set in $\mathbb{C}$.
\end{proof}

\begin{remark}
 Comparing the Corollary \ref{2.3} with the Theorem \ref{2.4} it can be observed that the Berezin range of $A$ is the limit of the Berezin range of corresponding operator $A_n$. This follows from the fact that
 $$ \lim_{n\longrightarrow \infty} \sqrt[n]{\frac{1}{(n+1)}}\left(\frac{n}{n+1}\right) =1.$$
\end{remark}

Here, we prove the general case of Theorem \ref{2.2} by considering an arbitrary element $g_i\in H^2(\mathbb{D})$.

\begin{theorem}
Let $A(f) =\sum_{i=1}^{n} \langle f,g_i \rangle g_i$, $f,g_i \in H^2(\mathbb{D})$ be a finite rank  operator on  $H^2(\mathbb{D})$. Then the Berezin range of $A$, $\text{Ber}(A)$ is convex in $\mathbb{C}$.
\end{theorem}
\begin{proof}
For $\lambda \in \mathbb{D}$, $ A(k_{\lambda}) = \sum_{i=1}^{n} \overline{g_i(\lambda)}g_i(z).$
The Berezin transform at $\lambda$ is
$$\widetilde{A}(\lambda)=(1-|\lambda|^2)\langle \sum_{i=1}^{n} \overline{g_i(\lambda)}g_i(z),{k_{\lambda}}\rangle=(1-|\lambda|^2)\sum_{i=1}^{n} |g_i(\lambda)|^{2},$$
 where $\lambda\in \mathbb{D}$. Given that $g_i \in H^2(\mathbb{D})$ is a holomorphic function, it implies that $\widetilde{A}(\lambda)$ is a real continuous function in $\mathbb{D}$. Since $\mathbb{D}$ is a connected set, the image of $\mathbb{D}$ under $\widetilde{A}$ is a connected set in $\mathbb{R}$. In the real line, $\mathbb{R}$, the only connected subsets are intervals and singleton sets. Thus the range of $\widetilde{A}(\lambda)$ must be an interval or singleton sets in $\mathbb{R}$. Therefor, the range of the Berezin transform of the operator $A$, $\text{Ber}(A)$ is a convex set in $\mathbb{C}$.
\end{proof}

Now, consider an arbitrary finite rank operator on the Hardy space $H^2(\mathbb{D})$,
$$A(f) = \sum_{i=1}^{n}\langle f,g_i \rangle h_i,$$
 $f,g_i,h_i \in H^2(\mathbb{D})$ and $i \in \{1,2,\cdots,n\}.$ Then we have $ A(k_{\lambda})  = \sum_{i=1}^{n} \overline{g_i(\lambda)}h_i(z).$ The Berezin transform of $A$ at $\lambda$ is given by
$$\widetilde{A}(\lambda) =(1-|\lambda|^2)\langle \sum_{i=1}^{n} \overline{g_i(\lambda)}h_i(z),{k_{\lambda}}\rangle =(1-|\lambda|^2)\sum_{i=1}^{n} \overline{g_i(\lambda)}h_i(\lambda),$$
 where $\lambda\in \mathbb{D}$.

 Using the expression for the Berezin transform, we now prove some geometric properties of the Berezin range of the finite rank operator $A$.

\begin{proposition}\label{2.8}
Let $A(f) =\sum_{i=1}^{n} \langle f,g_i \rangle h_i$, $f,g_i,h_i \in H^2(\mathbb{D})$ be a finite rank  operator on $H^2(\mathbb{D})$ with $g_i(z)= \sum_{m=1}^{\infty}a_{i,m}z^m$ and $h_i(z) = \sum_{m=1}^{\infty}b_{i,m}z^m$. If the coefficients of $g_i$ and $h_i$ are real,  then the Berezin range of $A$, $\text{Ber}(A)$ is closed under complex conjugation and therefore symmetric about the real axis.
 \end{proposition}
 \begin{proof}
 Let $A(f) =\sum_{i=1}^{n} \langle f,g_i \rangle h_i$ be a finite rank  operator on $H^2(\mathbb{D})$. Then the Berezin transform of $A$ at $\lambda$ is
 $$\widetilde{A}(\lambda) = (1-|\lambda|^2)\sum_{i=1}^{n} \overline{g_i(\lambda)}h_i(\lambda),$$
 where $\lambda\in \mathbb{D}.$ Substituting $g_i(\lambda)= \sum_{m=1}^{\infty}a_{i,m}\lambda^m$ and $h_i(\lambda) = \sum_{m=1}^{\infty}b_{i,m}\lambda^m$, we have
 $$\widetilde{A}(\lambda) = (1-|\lambda|^2)\sum_{i=1}^{n}\left( \overline{\sum_{m=1}^{\infty}a_{i,m}\lambda^m}\sum_{m=1}^{\infty}b_{i,m}\lambda^m\right).$$
 Put $\lambda=re^{i\theta}$. Then
 $$\widetilde{A}(re^{i\theta}) = (1-r^2)\sum_{i=1}^{n}\left( \overline{\sum_{m=1}^{\infty}a_{i,m}(re^{i\theta})^m}\sum_{m=1}^{\infty}b_{i,m}(re^{i\theta})^m\right).$$
 We claim that $\widetilde{A}(re^{i\theta}) =\overline{\widetilde{A}(re^{-i\theta})}$. So we compute $\overline{\widetilde{A}(re^{-i\theta})}$. Since $a_{i,m},~b_{i,m} \in \mathbb{R}$, we have
 \begin{center}
  	\begin{equation*}
  		\begin{split}
  			\overline{\widetilde{A}(re^{-i\theta})} &=\overline{(1-r^2)\sum_{i=1}^{n}\left( \overline{\sum_{m=1}^{\infty}a_{i,m}(re^{-i\theta})^m}\sum_{m=1}^{\infty}b_{i,m}(re^{-i\theta})^m\right)}\\
  			&=(1-r^2)\sum_{i=1}^{n}\left(\sum_{m=1}^{\infty}a_{i,m}(re^{-i\theta})^m\overline{\sum_{m=1}^{\infty}b_{i,m}(re^{-i\theta})^m}\right)\\
  			&=(1-r^2)\sum_{i=1}^{n}\left( \overline{\sum_{m=1}^{\infty}a_{i,m}(re^{i\theta})^m}\sum_{m=1}^{\infty}b_{i,m}(re^{i\theta})^m\right)\\
  			&= \widetilde{A}(re^{i\theta}).
  		\end{split}
  	\end{equation*}
  \end{center}
 \end{proof}
 
 \begin{corollary}\label{real}
Let $A(f) =\sum_{i=1}^{n} \langle f,g_i \rangle h_i$, $f,g_i,h_i \in H^2(\mathbb{D})$ be a finite rank  operator on $H^2(\mathbb{D})$. Let $g_i(z)= \sum_{m=1}^{\infty}a_{i,m}z^m$ and $h_i(z) = \sum_{m=1}^{\infty}b_{i,m}z^m$ with $a_{i,m},~b_{i,m} \in \mathbb{R}$. Then the following statement holds. If the Berezin range of $A$ on $H^2(\mathbb{D})$ is convex, then $Re\{\widetilde{A}(z)\} \in \textit{Ber}(A)$ for each $z\in \mathbb{D}$.
\end{corollary}

\begin{proof}
	Suppose $\textit{Ber}(A)$ is convex. Then from Proposition \ref{2.8} $\textit{Ber}(A)$ is closed under complex conjugation. Therefore we have
		$$\frac{1}{2}\widetilde{A}(z) + \frac{1}{2}\overline{\widetilde{A}(z)} = Re\{\widetilde{A}(z)\} \in \textit{Ber}(A).$$
\end{proof}
Now, we give an example of a rank one operator whose Berezin range is not convex.
\begin{example}
Let $A(f) = \langle f,1+z \rangle (1+z^2),$ where $f \in H^2(\mathbb{D})$ be a rank one operator on $ H^2(\mathbb{D})$. For $\lambda\in \mathbb{D}$, $ A(k_{\lambda}) = (1+\overline{\lambda})(1+z^2)$. The Berezin transform of $A$ at $\lambda$ is given by
 $$\widetilde{A}(\lambda) =(1-|\lambda|^2)\langle (1+\overline{\lambda})(1+z^2),k_{\lambda} \rangle = (1-|\lambda|^2)(1+\overline{\lambda})(1+\lambda^2).$$
 Put $\lambda = x+iy.$ Then, we have
 $$ \mathcal{R}(\widetilde{A}(\lambda)) = (1-x^2-y^2)\left[(1+x)(1+x^2-y^2) + 2xy^2\right]$$
 and
 $$ \mathcal{I}(\widetilde{A}(\lambda)) = (1-x^2-y^2)y\left[(1+x)2x + y^2- x^2-1\right].$$
 Our claim is to prove that $\textit{Ber}(A)$ is not convex. Suppose $\textit{Ber}(A)$ is convex. Then by Corollary \ref{real}, we have $\mathcal{R}\{\widetilde{A}(z)\} \in \textit{Ber}(A).$
 
 Therefore, for each $z\in\mathbb{D}$, we can find $\lambda\in\mathbb{D}$ such that
 $$\widetilde{A}(\lambda)= \mathcal{R}\{\widetilde{A}(z)\}.$$
 In turn, $\mathcal{I}\{\widetilde{A}(\lambda)\}= (1-x^2-y^2)y\left[(1+x)2x + y^2- x^2-1\right] = 0.$ Since $(1-x^2-y^2)>0$, the condition $\mathcal{I}\{\widetilde{A}(\lambda)\}=0$ holds either when $y=0$ or when $(x+1)^2 + y^2 =2$. 
 Suppose $y=0$. Then $\lambda\in (-1,1)$. Then
 $$ \mathcal{R}(\widetilde{A}(\lambda)) = (1-x^2)(1+x)(1+x^2)>0.$$
 Now, suppose that $(x+1)^2 + y^2 =2$. This is a circle with centre at $(-1,0)$ and radius $\sqrt{2}$. The arc of this circle that lies within $\mathcal{D}$ is situated in the first and fourth quadrants of the plane, intersecting the Y-axis at the points $(0,1)$ and $(0,-1)$. This implies that $x > 0.$ Therefore, 
   	\begin{equation*}
  		\begin{split}
  			\mathcal{R}(\widetilde{A}(\lambda)) &= (1-x^2-y^2)\left[(1+x)(1+x^2-y^2) + 2xy^2\right]\\
  			&=(1-x^2-y^2)(1+x^2-y^2 +x +x(x^2 +y^2))\\
  			&=(1-x^2-y^2)(1+x^2-y^2 +x +x(1-2x))\\
  			&=(1-x^2-y^2)(1-(x^2+y^2) +2x)\\
  			&=(1-x^2-y^2)(1-(1-2x) +2x)\\
  			&=(1-x^2-y^2)4x > 0.
  		\end{split}
  	\end{equation*}
  So, if $ \mathcal{I}(\widetilde{A}(\lambda))=0$, then $\widetilde{A}(\lambda) >0.$ But for $\lambda = \frac{-2}{10} +i\frac{9}{10}$,
  $$\mathcal{R}(\widetilde{A}(\lambda)) < 0.$$
  This is a contradiction to the Corollary \ref{real}. Therefore, $\textit{Ber}(A)$ is not convex.
\end{example}
 \begin{figure}[h]
  	\caption{$\text{Ber}(A)$ on $H^2(\mathbb{D})$ for $\lambda = \frac{-2}{10} +i\frac{9}{10}$(apparently not convex).}
  	\includegraphics[scale=0.5]{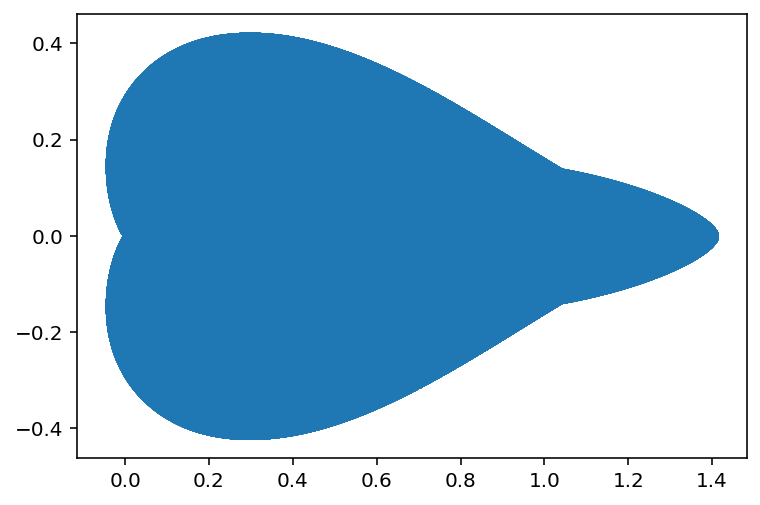}
  	\label{fig1}
  	 \end{figure}
  	  The following theorem calculate the Berezin range of a rank one operator of the form $A(f) = \langle f,g \rangle h$, where $f,g,h \in H^2(\mathbb{D})$.
 \begin{theorem}
 Let $A(f) = \langle f,z^n \rangle z^m$, $m>n$ be a rank one operator on $H^2(\mathbb{D})$. Then the Berezin range of $A$, $\text{Ber}(A)$ is a disc with centre at the origin and radius $\left(\frac{2}{m+n+2}\right)\left(\frac{m+n}{m+n+2}\right)^{\frac{m+n}{2}}$ and therefore is convex in $\mathbb{C}$.
 \end{theorem}
 \begin{proof}
 For $\lambda\in \mathbb{D}$, $ A(k_{\lambda}) = \overline{\lambda}^nz^m$. The Berezin transform of $A$ at $\lambda$ is given by
 $$\widetilde{A}(\lambda) =(1-|\lambda|^2)\langle \overline{\lambda}^nz^m,k_{\lambda} \rangle = (1-|\lambda|^2)|\lambda|^{2n}\lambda^{m-n}.$$
  Put $\lambda=re^{i\theta}$. Then we get
  $$\widetilde{A}(\lambda)= (1-r^2)r^{2n}r^{m-n}e^{i\theta(m-n)} =(r^{m+n}-r^{m+n+2})e^{i\theta(m-n)}.$$
  Therefore the Berezin range of $A$,
  $$\text{Ber}(A)=\left\lbrace (r^{m+n}-r^{m+n+2})e^{i\theta(m-n)}: re^{i\theta}\in \mathbb{D} \right\rbrace.$$
  For each $r\in [0,1)$, this set is a circular set. For $\eta \in \text{Ber}(A)$, then we have $\eta = (r^{m+n}-r^{m+n+2})e^{i\theta(m-n)}$ for some $re^{i\theta}\in \mathbb{D}$. Then for any $t\in [0,2\pi)$, 
  $$\eta e^{it} = (r^{m+n}-r^{m+n+2})e^{i\theta(m-n)} e^{it} = (r^{m+n}-r^{m+n+2})e^{i(\theta(m-n)+t)} \in \text{Ber}(A),$$
  since $\eta e^{it}$ is the image of $re^{i(\theta(m+n)+t)}\in \mathbb{D}$ for all $t\in [0,2\pi)$. Since $0\in \text{Ber}(A)$, it is easy to observe that $\text{Ber}(A)$ is a disc with centre at the origin and radius as $\sup_{r\in [0,1)}(r^{m+n}-r^{m+n+2})$. 
  
  To find the extreme points of $(r^{m+n}-r^{m+n+2})$, we differentiate it and equate the derivative to zero. We get,
$$0 = (m+n)r^{m+n-1} - (m+n+2)r^{m+n+1}= r^{m+n-1}\left((m+n) - (m+n+2)r^2\right).$$
  Therefore the extreme points are $r=0$ and $r=\sqrt{\frac{m+n}{m+n+2}}$. Substituting this value we get the radius of the disc as $\left(\frac{2}{m+n+2}\right)\left(\frac{m+n}{m+n+2}\right)^{\frac{m+n}{2}}$. Hence, $$\text{Ber}(A) = \mathbb{D}_{\left(\frac{2}{m+n+2}\right)\left(\frac{m+n}{m+n+2}\right)^{\frac{m+n}{2}}},$$ which is convex in $\mathbb{C}$.
 \end{proof}
 \begin{example}
Let $A(f) = \langle f,z \rangle z^2,$ where $f \in H^2(\mathbb{D})$ be a rank one operator on $ H^2(\mathbb{D})$. 
Then, $\text{Ber}(A) = \mathbb{D}_{\sqrt{\frac{3}{5}}\frac{6}{25}}$, which is convex in $\mathbb{C}$.
 \end{example}
\section{Finite rank operators on Bergman space}
Consider the open unit disc $\mathbb{D}$ in the complex plane $\mathbb{C}$  and let $\text{Hol}(\mathbb{D})$ be the collection of holomorphic functions on $\mathbb{D}$.. The Bergman space on the unit disc, $A^2(\mathbb{D})$, is
$$A^2(\mathbb{D}):= \left\{ f \in\text{Hol}(\mathbb{D}) : \int_{\Omega} |f(z)|^2 dV(z) < \infty\right\},$$
where $dV$ is the normalized area measure on $\mathbb{D}$. The inner product for the elements $f,g\in A^2(\mathbb{D})$ is given by 
$$\langle f, g \rangle_{A^2(\mathbb{D})} = \int_{\mathbb{D}} f(z) \overline{g(z)} \, dV(z).$$
$A^2(\mathbb{D})$ is a reproducing kernel Hilbert space with $\|f\|_{A^2(\mathbb{D})} =\left( \int_{\mathbb{D}} |f(z)|^2 \, dV(z) \right)^{\frac{1}{2}}.$ One can see that the reproducing kernel for $A^2(\mathbb{D})$ at $\lambda$ is given by
$$ k_{\lambda}(z) = \frac{1}{(1-\bar{\lambda}z)^2},$$
where $z,\lambda \in\mathbb{D}$

We will now present analogous results for the Bergman space, similar to those established for the Hardy space in the previous section. We will start with an example of a rank one operator
acting on the Bergman space $A^2(\mathbb{D})$.
\begin{example}\label{eg}
Consider the rank one operator $A(f) = \langle f,z \rangle z$ acting on the Bergman space $A^2(\mathbb{D})$. For $\lambda \in \mathbb{D}$, we have 
$$ A(k_{\lambda}) = \langle k_{\lambda},z \rangle z = \overline{\lambda}z.$$
The Berezin transform at $\lambda$ is
$$\widetilde{A}(\lambda) =(1-|\lambda|^2)^2 |\lambda|^2.$$
Thus $\text{Ber}(A) = [0,\frac{4}{27 }]$, which is a convex set in $\mathbb{C}$. 
\end{example}
Next, we calculate the Berezin range of the rank one operator, \( A(f) = \langle f, z^n \rangle z^n \) for \( n \in \mathbb{N} \).

\begin{proposition}\label{4.2}
Let $A(f) = \langle f,z^n \rangle z^n$ be a rank one operator on $A^2(\mathbb{D})$. Then the Berezin range of $A$ is $\text{Ber}(A)=\left[0,\frac{4n^n}{(n+2)^{n+2}}\right]$, which is convex in $\mathbb{C}$.
\end{proposition}
\begin{proof}
For $\lambda \in \mathbb{D}$, $ A(k_{\lambda}) = \overline{\lambda}^nz^n.$
The Berezin transform at $\lambda$ is
  	\begin{equation*}
  		\begin{split}
  			\widetilde{A}(\lambda) &=(1-|\lambda|^2)^2\langle \overline{\lambda}^nz^n,{k_{\lambda}}\rangle\\
  			&=(1-|\lambda|^2)^2 |\lambda|^{2n}\\
  			&=(1-2|\lambda|^2 + |\lambda|^4)|\lambda|^{2n}\\
  			&=|\lambda|^{2n}- 2|\lambda|^{2n+2}+|\lambda|^{2n+4}
  		\end{split}
  	\end{equation*}
 where $|\lambda|\in [0,1)$. Observe that  $\widetilde{A}(\lambda) = \widetilde{A}(|\lambda|)$ is a real function. Now we differentiate it with respect to $|\lambda|$ and equate it to zero to find the extreme points.
$$0 =2n|\lambda|^{2n-1} - (4n+4)|\lambda|^{2n+1}+(2n+4)|\lambda|^{2n+3}=2|\lambda|^{2n-1}\left(n-(2n+2)|\lambda|^{2}+ (n+2)|\lambda|^4\right).$$
  This happens if and only if $|\lambda|=0$ or $|\lambda|^2=\frac{n}{n+2}$. Now if $|\lambda|=0$, then $\widetilde{A}(\lambda) = 0$. If $|\lambda|^2=\frac{n}{n+2}$, then
$$\widetilde{A}(\lambda) = \left(1-\frac{n}{n+2}\right)^2 \left(\frac{n}{n+2}\right)^{n}=\frac{4n^n}{(n+2)^{n+2}}.$$
  Thus the Berezin range of $A$, $\text{Ber}(A)= \left[0,\frac{4n^n}{(n+2)^{n+2}}\right]$, which is a convex set in $\mathbb{C}$.
\end{proof}

  \begin{theorem}\label{4.3}
  Let $ A_n(f) = \sum_{i=1}^{n}\langle f,a_iz^i \rangle a_iz^i$ be a finite rank operator on $A^2(\mathbb{D})$, where $a_i \in \mathbb{C}$. Then the Berezin range of $A_n$ is convex in $\mathbb{C}$.
  \end{theorem}
  \begin{proof}
For $\lambda \in \mathbb{D}$, $A_n(k_{\lambda}) = \sum_{i=1}^{n} |a_i|^2\overline{\lambda}^iz^i.$
The Berezin transform of $A_n$ at $\lambda$ is
  	\begin{equation*}
  		\begin{split}
  			\widetilde{A_n}(\lambda) &=(1-|\lambda|^2)^2\langle \sum_{i=1}^{n} |a_i|^2\overline{\lambda}^iz^i,{k_{\lambda}}\rangle\\
  			&=(1-|\lambda|^2)^2 \sum_{i=1}^{n} |a_i|^2|\lambda|^{2i}.
  		\end{split}
  	\end{equation*}
  Thus, the Berezin transform is a real continuous function in the complex plane $\mathbb{C}$. Therefore $\text{Ber}(A_n)$ is convex.
  \end{proof} 
  
  \begin{remark}
  We can observe that the Berezin transform reaches $0$ as $\lambda=0$ and $|\lambda|\longrightarrow 1$. The exact maximum value occurs somewhere in $(0,1)$ and it depends on the specific values of $a_i$. Thus the $\text{Ber}(A_n)=\left[0,\text{max}~\widetilde{A_n}(\lambda)\right]$
  \end{remark}

Now, we prove the general case of  Theorem \ref{4.3} by considering an arbitrary element $g_i\in A^2(\mathbb{D})$.

\begin{theorem}
Let $A(f) =\sum_{i=1}^{n} \langle f,g_i \rangle g_i$, $f,g_i \in A^2(\mathbb{D})$ be a finite rank  operator on the reproducing kernel Hilbert space $A^2(\mathbb{D})$. Then the Berezin range of $A$, $\text{Ber}(A)$ is convex in $\mathbb{C}$.
\end{theorem}
\begin{proof}
For $\lambda \in \mathbb{D}$, $ A(k_{\lambda}) = \sum_{i=1}^{n} \overline{g_i(\lambda)}g_i(z).$
The Berezin transform at $\lambda$ is
  	\begin{equation*}
  		\begin{split}
  			\widetilde{A}(\lambda) &=(1-|\lambda|^2)^2\langle \sum_{i=1}^{n} \overline{g_i(\lambda)}g_i(z),{k_{\lambda}}\rangle\\
  			&=(1-|\lambda|^2)^2\sum_{i=1}^{n} |g_i(\lambda)|^{2}
  		\end{split}
  	\end{equation*}
 where $\lambda\in \mathbb{D}$. Given that $g_i \in A^2(\mathbb{D})$ is a holomorphic function, it implies that $\widetilde{A}(\lambda)$ is continuous function of $\lambda$. So $\widetilde{A}(\lambda)$ is a real continuous function in $\mathbb{D}$. Since $\mathbb{D}$ is a connected set and $\widetilde{A}(\lambda)$ is continuous on $\mathbb{D}$, the image of $\mathbb{D}$ under $\widetilde{A}$ must also be a connected set in $\mathbb{R}$. In the real line, $\mathbb{R}$, the only connected subsets are intervals including degenerate intervals consisting of a single point. Therefore the range of $\widetilde{A}(\lambda)$ must be an interval in $\mathbb{R}$. Since intervals in $\mathbb{R}$ are convex sets, the range of the Berezin transform of the operator $A$, $\text{Ber}(A)$ is a convex set in $\mathbb{C}$.
\end{proof}

Consider an arbitrary finite rank operator in the Bergman space $A^2(\mathbb{D})$,
$$A(f) = \sum_{i=1}^{n}\langle f,g_i \rangle h_i,$$
 $f,g_i,h_i \in A^2(\mathbb{D})$ and $i \in \{1,2,\cdots,n\}.$ Then we have 
$$ A(k_{\lambda}) = \sum_{i=1}^{n}\langle k_{\lambda},g_i(z) \rangle h_i(z) = \sum_{i=1}^{n} \overline{g_i(\lambda)}h_i(z),$$
where $k_{\lambda}$ is the reproducing kernel at $\lambda$. The Berezin transform of $A$ at $\lambda$ is given by
  	\begin{equation*}
  		\begin{split}
  			\widetilde{A}(\lambda) &=(1-|\lambda|^2)^2\langle \sum_{i=1}^{n} \overline{g_i(\lambda)}h_i(z),{k_{\lambda}}\rangle\\
  			&=(1-|\lambda|^2)^2\sum_{i=1}^{n} \overline{g_i(\lambda)}h_i(\lambda)
  		\end{split}
  	\end{equation*}
 where $\lambda\in \mathbb{D}$.
 
 Using the expression for the Berezin transform, we prove some geometric properties of the Berezin range of the finite rank operator $A$.

\begin{proposition}\label{4.6}
Let $A(f) =\sum_{i=1}^{n} \langle f,g_i \rangle h_i$, $f,g_i,h_i \in A^2(\mathbb{D})$ be a finite rank  on $A^2(\mathbb{D})$ with $g_i(z)= \sum_{m=1}^{\infty}a_{i,m}z^m$ and $h_i(z) = \sum_{m=1}^{\infty}b_{i,m}z^m$. If the coefficients of $g_i$ and $h_i$ are real, then the Berezin range of $A$, $\text{Ber}(A)$ is closed under complex conjugation and therefore symmetric about the real axis.
 \end{proposition}
 \begin{proof}
 Let $A(f) =\sum_{i=1}^{n} \langle f,g_i \rangle h_i$ be a finite rank  operator on $A^2(\mathbb{D})$. Then the Berezin transform of $A$ at $\lambda$ is
 $$\widetilde{A}(\lambda) = (1-|\lambda|^2)^2\sum_{i=1}^{n} \overline{g_i(\lambda)}h_i(\lambda),$$
 where $\lambda\in \mathbb{D}.$ Substituting $g_i(\lambda)= \sum_{m=1}^{\infty}a_{i,m}\lambda^m$ and $h_i(\lambda) = \sum_{m=1}^{\infty}b_{i,m}\lambda^m$, we have
 $$\widetilde{A}(\lambda) = (1-|\lambda|^2)^2\sum_{i=1}^{n}\left( \overline{\sum_{m=1}^{\infty}a_{i,m}\lambda^m}\sum_{m=1}^{\infty}b_{i,m}\lambda^m\right).$$
 Put $\lambda=re^{i\theta}$. Then
 $$\widetilde{A}(re^{i\theta}) = (1-r^2)^2\sum_{i=1}^{n}\left( \overline{\sum_{m=1}^{\infty}a_{i,m}(re^{i\theta})^m}\sum_{m=1}^{\infty}b_{i,m}(re^{i\theta})^m\right).$$
 We claim that $\widetilde{A}(re^{i\theta}) =\overline{\widetilde{A}(re^{-i\theta})}$. So we compute $\overline{\widetilde{A}(re^{-i\theta})}$. Since $a_{i,m},~b_{i,m} \in \mathbb{R}$, we have
 \begin{center}
  	\begin{equation*}
  		\begin{split}
  			\overline{\widetilde{A}(re^{-i\theta})} &=\overline{(1-r^2)^2\sum_{i=1}^{n}\left( \overline{\sum_{m=1}^{\infty}a_{i,m}(re^{-i\theta})^m}\sum_{m=1}^{\infty}b_{i,m}(re^{-i\theta})^m\right)}\\
  			&=(1-r^2)^2\sum_{i=1}^{n}\left(\sum_{m=1}^{\infty}a_{i,m}(re^{-i\theta})^m\overline{\sum_{m=1}^{\infty}b_{i,m}(re^{-i\theta})^m}\right)\\
  			&=(1-r^2)^2\sum_{i=1}^{n}\left( \overline{\sum_{m=1}^{\infty}a_{i,m}(re^{i\theta})^m}\sum_{m=1}^{\infty}b_{i,m}(re^{i\theta})^m\right)\\
  			&= \widetilde{A}(re^{i\theta}).
  		\end{split}
  	\end{equation*}
  \end{center}
 \end{proof}
 
 \begin{corollary}\label{rea}
Let $A(f) =\sum_{i=1}^{n} \langle f,g_i \rangle h_i$, $f,g_i,h_i \in A^2(\mathbb{D})$ be a finite rank  operator on $A^2(\mathbb{D})$. Let $g_i(z)= \sum_{m=1}^{\infty}a_{i,m}z^m$ and $h_i(z) = \sum_{m=1}^{\infty}b_{i,m}z^m$ with $a_{i,m},~b_{i,m} \in \mathbb{R}$. Then the following statement holds. If the Berezin range of $A$ on $A^2(\mathbb{D})$ is convex, then $Re\{\widetilde{A}(z)\} \in \textit{Ber}(A)$ for each $z\in \mathbb{D}$.
\end{corollary}

\begin{proof}
	Suppose $\textit{Ber}(A)$ is convex. Then from Proposition \ref{4.6} $\textit{Ber}(A)$ is closed under complex conjugation. Therefore we have
		$$\frac{1}{2}\widetilde{A}(z) + \frac{1}{2}\overline{\widetilde{A}(z)} = Re\{\widetilde{A}(z)\} \in \textit{Ber}(A)).$$
\end{proof}
The following example gives an example of a rank one operator on the Bergman space whose Berezin range is not convex.
\begin{example}
Let $A(f) = \langle f,1+z \rangle (1+z^2),$ where $f \in A^2(\mathbb{D})$ be a rank one operator on $ A^2(\mathbb{D})$. For $\lambda\in \mathbb{D}$, $ A(k_{\lambda}) = (1+\overline{\lambda})(1+z^2)$. The Berezin transform of $A$ at $\lambda$ is given by
 $$\widetilde{A}(\lambda) =(1-|\lambda|^2)^2\langle (1+\overline{\lambda})(1+z^2),k_{\lambda} \rangle = (1-|\lambda|^2)^2(1+\overline{\lambda})(1+\lambda^2).$$
 Put $\lambda = x+iy.$ Then, we have
 $$ \mathcal{R}(\widetilde{A}(\lambda)) = (1-x^2-y^2)^2\left[(1+x)(1+x^2-y^2) + 2xy^2\right]$$
 and
 $$ \mathcal{I}(\widetilde{A}(\lambda)) = (1-x^2-y^2)^2y\left[(1+x)2x + y^2- x^2-1\right].$$
 Our claim is to prove that $\textit{Ber}(A)$ is not convex. Suppose $\textit{Ber}(A)$ is convex. Then by Corollary \ref{rea}, we have $\mathcal{R}\{\widetilde{A}(z)\} \in \textit{Ber}(A).$
 
 Therefore, for each $z\in\mathbb{D}$, we can find $\lambda\in\mathbb{D}$ such that
 $$\widetilde{A}(\lambda)= \mathcal{R}\{\widetilde{A}(z)\}.$$
 In turn, $\mathcal{I}\{\widetilde{A}(\lambda)\}= (1-x^2-y^2)^2y\left[(1+x)2x + y^2- x^2-1\right] = 0.$ Since $(1-x^2-y^2)^2>0$, the condition $\mathcal{I}\{\widetilde{A}(\lambda)\}=0$ holds either when $y=0$ or when $(x+1)^2 + y^2 =2$. 
 Suppose $y=0$. Then $\lambda\in (-1,1)$. Then
 $$ \mathcal{R}(\widetilde{A}(\lambda)) = (1-x^2)^2(1+x)(1+x^2)>0.$$
 Now, suppose that $(x+1)^2 + y^2 =2$. This is a circle with centre at $(-1,0)$ and radius $\sqrt{2}$. The arc of this circle that lies within $\mathcal{D}$ is situated in the first and fourth quadrants of the plane, intersecting the Y-axis at the points $(0,1)$ and $(0,-1)$. This implies that $x > 0.$ Therefore, 
   	\begin{equation*}
  		\begin{split}
  			\mathcal{R}(\widetilde{A}(\lambda)) &= (1-x^2-y^2)^2\left[(1+x)(1+x^2-y^2) + 2xy^2\right]\\
  			&=(1-x^2-y^2)^2(1+x^2-y^2 +x +x^3 +xy^2)\\
  			&=(1-x^2-y^2)^2(1+x^2-y^2 +x +x(x^2 +y^2))\\
  			&=(1-x^2-y^2)^2(1+x^2-y^2 +x +x(1-2x))\\
  			&=(1-x^2-y^2)^2(1-(x^2+y^2) +2x)\\
  			&=(1-x^2-y^2)^2(1-(1-2x) +2x)\\
  			&=(1-x^2-y^2)^24x > 0.
  		\end{split}
  	\end{equation*}
  So, if $ \mathcal{I}(\widetilde{A}(\lambda))=0$, then $\widetilde{A}(\lambda) >0.$ But for $\lambda = \frac{-2}{10} +i\frac{9}{10}$,
  $$\mathcal{R}(\widetilde{A}(\lambda)) < 0.$$
  This is a contradiction to the Corollary \ref{rea}. Therefore, $\textit{Ber}(A)$ is not convex.
\end{example}
\begin{figure}[h]
  	\caption{$\text{Ber}(A)$ on $A^2(\mathbb{D})$ for $\lambda = \frac{-2}{10} +i\frac{9}{10}$(apparently not convex).}
  	\includegraphics[scale=0.5]{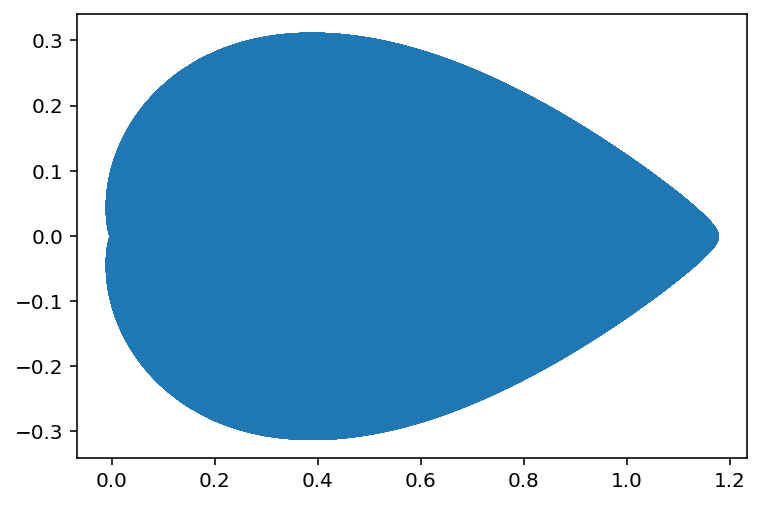}
  	\label{fig2}
  	 \end{figure}
  The next theorem characterize the convexity of the Berezin range of a rank one operator on the Bergaman space.
 \begin{theorem}
 Let $A(f) = \langle f,z^n \rangle z^m$, $m>n$ be a rank one operator on  $A^2(\mathbb{D})$. Then the Berezin range of $A$, $\text{Ber}(A)$ is a disc with centre at the origin and radius $\left(\frac{2n}{m+3n}\right)^2\left(\frac{m+n}{m+3n}\right)^{\frac{m+n}{2}}$ and therefore is convex in $\mathbb{C}$.
 \end{theorem}
 \begin{proof}
 For $\lambda\in \mathbb{D}$, $ A(k_{\lambda}) = \overline{\lambda}^nz^m.$ The Berezin transform of $A$ at $\lambda$ is given by
  	\begin{equation*}
  		\begin{split}
  			\widetilde{A}(\lambda) &=(1-|\lambda|^2)^2\langle \overline{\lambda}^nz^m,k_{\lambda} \rangle\\
  			&= (1-|\lambda|^2)^2|\lambda|^{2n}\lambda^{m-n}.
  		\end{split}
  	\end{equation*}
  Put $\lambda=re^{i\theta}$. Then we get
  $$\widetilde{A}(\lambda)= (1-r^2)^2r^{2n}r^{m-n}e^{i\theta(m-n)} =(r^{m+n}+r^{m+n+4}-2r^{m+n+2})e^{i\theta(m-n)}.$$
  Therefore the Berezin range of $A$,
  $$\text{Ber}(A)=\left\lbrace (r^{m+n}+r^{m+n+4}-2r^{m+n+2})e^{i\theta(m-n)}: re^{i\theta}\in \mathbb{D} \right\rbrace.$$
  For each $r\in [0,1)$, this set is a circular set. For $\eta \in \text{Ber}(A)$, we have $\eta = (r^{m+n}+r^{m+n+4}-2r^{m+n+2})e^{i\theta(m-n)}$ for some $re^{i\theta}\in \mathbb{D}$. Then for any $t\in [0,2\pi]$, we have
  \begin{center}
  \begin{equation*}
  \begin{split}
  \eta e^{it} &= (r^{m+n}+r^{m+n+4}-2r^{m+n+2})e^{i\theta(m-n)} e^{it}\\ 
  &= (r^{m+n}+r^{m+n+4}-2r^{m+n+2})e^{i(\theta(m-n)+t)}.
  \end{split}
  \end{equation*}
  \end{center}
  
 Therefore, we get $\eta e^{it} \in \text{Ber}(A)$, as $\eta e^{it}$ is the image of $re^{i(\theta(m+n)+t)}\in \mathbb{D}$ for all $t\in [0,2\pi]$. Since $0\in \text{Ber}(A)$, it is easy to observe that $\text{Ber}(A)$ is a disc with centre at the origin and $\sup_{r\in [0,1)}(r^{m+n}+r^{m+n+4}-2r^{m+n+2})$ as radius. To find its extreme points of $(r^{m+n}+r^{m+n+4}-2r^{m+n+2})$, we differentiate it and equate the derivative to zero. So we get,
  \begin{center}
  	\begin{equation*}
  		\begin{split}
  			0 &= (m+n)r^{m+n-1} + (m+n+4)r^{m+n+3} - 2(m+n+2)r^{m+n+1}\\
  			&= r^{m+n-1}\left((m+n) +(m+n+4)r^4 - 2(m+n+2)r^2\right).
  		\end{split}
  	\end{equation*}
  \end{center}
  Therefore the extreme points are $r=0$ and $r=\sqrt{\frac{m+n}{m+n+4}}$. Substituting this value we get the radius of the disc as $\left(\frac{4}{m+n+4}\right)^2\left(\frac{m+n}{m+n+4}\right)^{\frac{m+n}{2}}$. Hence, $\text{Ber}(A) = \mathbb{D}_{\left(\frac{4}{m+n+4}\right)^2\left(\frac{m+n}{m+n+4}\right)^{\frac{m+n}{2}}}$, which is convex in $\mathbb{C}$.
 \end{proof}
  \begin{example}
 Consider the rank one operator 
 $$A(f) = \langle f,z \rangle z^2,$$ where $f \in A^2(\mathbb{D})$. Then, $\text{Ber}(A) = \mathbb{D}_{\sqrt{\frac{3}{7}}\frac{48}{343}}$, which is convex in $\mathbb{C}$.
 \end{example}
 \section{Reproducing kernel Hilbert space operator inequalities}
 In this section, we present applications of some scalar inequalities. This applicatios will treat inequalities, where Berezin radius, Berezin norm and operator means are discussed. For this, we need to recall some notions related to reproducing kernel Hilbert space operators. Let $\mathcal{H}(\Omega)$ be a reproducing kernel Hilbert space over some set $\Omega$, and let $B(\mathcal{H}(\Omega))$ be the $C^*$-algebra of all bounded linear operators on $\mathcal{H}(\Omega)$. For $T\in B(\mathcal{H}(\Omega))$, the Berezin norms and the Berezin radius are defined respectively as 
 $$\|T\|_{B,1} := \sup_{\lambda\in \Omega}\|A\hat{k}_\lambda\|, \quad \|T\|_{B,2} := \sup_{\lambda, \mu \in \Omega}\langle T \hat{k}_\lambda,\hat{k}_\mu\rangle$$
 and 
 $$\text{ber}(T) := \sup_{\lambda \in \Omega}|\langle T \hat{k}_\lambda,\hat{k}_\lambda\rangle|.$$
 Clearly,
 $$\text{ber}(T) \leq \|T\|_{B,2} \leq \|T\|_{B,1}.$$
 Also, $\text{ber}(T) \leq w(T)$ (numerical radius) for any $T\in B(\mathcal{H}(\Omega))$. We should remark that finding better bounds for Berezin norms and the Berezin radius has received a renowned interest in the last few years.
 
 Using the same approach, we will be able to find an inequality that relates the geometric mean of $|T|^{2\nu}$ and $|T|^{2(1-\nu)}$ with the Berezin radius, where $|T| = (T^*T)^{\frac{1}{2}}$. For this, we recall that the weighted geometric mean of the strictly positive operators $A,B \in B(\mathcal{H}(\Omega))$ is
$$A \#_{t} B = A^{1/2}\left(A^{-1/2}BA^{-1/2}\right)^t A^{1/2};~~ 0\leq t \leq 1.$$
When $t=\frac{1}{2}$, we write $\#$ instead of $\#_{1/2}$. Operator means and their inequalities have received a considerable attention in the literature, as one can find in \cite{11,4,3,2,20}.

Our results will make use of the angle $\angle_{x,y}$ between two vectors $x,y \in\mathbb{C}^n$ or $x,y \in\mathcal{H}(\Omega)$. For such $x,y$, the Cauchy-Schwarz inequality states that $|\langle x,y\rangle| \leq \|x\|\|y\|.$ From this, the angle between the non-zero vectors $x,y$ can be defined by
$$\angle_{x,y}:= \cos^{-1}\left(\frac{|\langle x,y\rangle|}{\|x\|\|y\|}\right).$$
\subsection{Scalar inequalities}
To establish our results, we need the following known inequalities (see \cite{20}).
\begin{proposition}\label{cd}
Let $c,d\in \mathbb{C}$ be two numbers.
Then
$$\left| \frac{c+d}{2}\right| \leq \int_{0}^{1} |sc+(1-s)d| \, ds \leq  \frac{|c|+|d|}{2}.
$$
\end{proposition}
\begin{proposition}\label{cos}
Let $x,y\in \mathcal{H}$. Then
$$|\langle x,y\rangle| \leq  \int_{0}^{1}|te^{i\theta} + (1-t)e^{-i\theta}| dt \|x\|\|y\| \leq  \|x\|\|y\|,$$
where $\theta = \angle_{x,y}$.
\end{proposition}
\begin{proof}
Indeed, we have
\begin{equation}\label{eq5.1}
|\langle x,y\rangle| = |\cos\theta|  \|x\|\|y\|.
\end{equation}
Since $\cos\theta = \mathcal{R}(e^{i\theta}) = \frac{e^{i\theta}+e^{-i\theta}}{2},$ Proposition \ref{cd} applied for $c=e^{i\theta}$ and $d=e^{-i\theta}$ implies that
$$|\langle x,y\rangle| \leq  \int_{0}^{1}|te^{i\theta} + (1-t)e^{-i\theta}| dt \|x\|\|y\| \leq  \|x\|\|y\|,$$
as desired.
\end{proof}
We remark that Proposition \ref{cos} means
$$|\cos\theta| \leq \int_{0}^{1}|te^{i\theta} + (1-t)e^{-i\theta}| dt \leq 1.$$
To calculate the constant that appears in Proposition \ref{cos}, notice that for an arbitrary $\theta \in \mathbb{R}$,
\begin{equation*}
\begin{split}
|te^{i\theta} + (1-t)e^{-i\theta}| &= |t(\cos \theta + i \sin \theta) + (1-t)(\cos \theta - i \sin \theta)|\\
&= |\cos \theta + i(2t-1) \sin \theta|\\
&=\sqrt{\cos^2 \theta + (2t-1)^2 \sin^2 \theta}
\end{split}
\end{equation*}
For the case $\sin \theta =0$, we have  $\int_{0}^{1} \sqrt{\cos^2 \theta} dt = |\cos \theta | = 1$.
If $\sin \theta \neq 0$, then
\begin{equation*}
\begin{split}
\int_{0}^{1} \sqrt{\cos^2 \theta + (2t-1)^2 \sin^2 \theta} dt &= \frac{|\sin \theta|}{2}\int_{-1}^{1} \sqrt{s^2+\cot^2 \theta} ds\\
&= \frac{|\sin \theta|}{4}\left[s \sqrt{s^2+\cot^2 \theta} + \cot^2\theta\log|s + \sqrt{s^2+\cot^2 \theta}|\right]_{-1}^{1} \\
&=\frac{1}{2} + \frac{1}{4} |\sin \theta|\cot^2\theta\log\left|\frac{1+|\sin \theta|}{1- |\sin \theta|}\right|.
\end{split}
\end{equation*}
We set $$\mu(\theta):=\frac{1}{2} + \frac{1}{4} |\sin \theta|\cot^2\theta\log\left|\frac{1+|\sin \theta|}{1- |\sin \theta|}\right|.$$
Hence
$$\int_{0}^{1} \sqrt{\cos^2 \theta + (2t-1)^2 \sin^2 \theta} dt=\mu(\theta).$$
Thus
$$\mu(\theta)=\frac{1}{4}\left(2 + \cos \theta\cot\theta\log\frac{1+\sin \theta}{1-\sin\theta}\right), ~\theta \neq n\pi,$$
where $n=0,1,2,\cdots$. Since $\mu(\theta) \to 1$ when $\theta \to n\pi$ and $|\cos\theta|=1$ for $\theta=n\pi$, where $n=0,1,2,\cdots$, we deduce that
$$\mu(\theta)=\frac{1}{4}\left(2 + \cos \theta\cot\theta\log\frac{1+\sin \theta}{1-\sin\theta}\right).$$
\begin{proposition}\label{mu}
The function $\mu(\theta)$ is decreasing on the interval $[0,\frac{\pi}{2}]$ and is increasing in $[\frac{\pi}{2},\pi]$.
\end{proposition} 
\begin{proof}
See \cite{20}.
\end{proof}
The following are straightforward consequences from Proposition \ref{mu}.
\begin{corollary}
The inequality $\frac{1}{2}\leq \mu(\theta)\leq 1$ holds for $\theta\geq 0$.
\end{corollary}
\begin{corollary}\label{5.5}
The following holds.
\begin{enumerate}
\item[(1)] If $0\leq \theta_1< \theta <\theta_2\leq \frac{\pi}{2}$, then $\mu(\theta)\leq \mu(\theta_1)$.
\item[(2)] If $\frac{\pi}{2}\leq \theta_1< \theta <\theta_2\leq \pi$, then $\mu(\theta)\leq \mu(\theta_2)$.
\end{enumerate}
\end{corollary}
The following reverse of the triangle inequality is proved in \cite{20}.
\begin{proposition}\label{5.5}
Let $c,d\in\mathbb{C}$. Then for any $t\in(0,1)$ we have
$$\frac{|c|+|d|}{2}-\frac{1}{2\tau_t}\left((1-t)|c|+t|d|-|(1-t)c+td|\right)\leq \left|\frac{c+d}{2}\right|,$$
where $\tau_t:=\min\{t,1-t\}$.
\end{proposition}
Remark that if $x,y\in\mathcal{H}$ and $\tau_t:=\min\{t,1-t\}$ with $0<t<1$, then by Proposition \ref{5.5}, we have
\begin{equation*}
\begin{split}
0 &\leq 1-\frac{1}{2\tau_t}\left(1-|te^{i\theta}+(1-t)e^{-i\theta}|\right)\\
&\leq \left|\frac{e^{i\theta}+e^{-i\theta}}{2}\right|=\cos \theta,
\end{split}
\end{equation*}
where $\theta = \angle_{x,y}$. Hence
\begin{equation}\label{eq5.2}
0\leq \gamma_t(\theta)\|x\| \|y\|\leq |\langle x,y\rangle|,
\end{equation}
where 
\begin{equation*}
\begin{split}
\gamma_t(\theta) &:= 1-\frac{1}{2\tau_t}\left(1-|te^{i\theta}+(1-t)e^{-i\theta}|\right)\\
&= 1-\frac{1}{2\tau_t}\left(1-\sqrt{\cos^2 \theta+(2t-1)^2\sin^2\theta}\right)
\end{split}
\end{equation*}
for $\theta\in[0,\pi]$ and $\tau_t:=\min\{t,1-t\}$ with $0<t<1$. This provides a reverse of the Cauchy-Schwarz inequality. 

The following is proved in \cite{20}.
\begin{proposition}\label{5.7}
Let $0<t<1$ be fixed, and let $\gamma_t$ be as above. Then $\gamma_t(\theta)$ is decreasing on $[0,\frac{\pi}{2}]$ and increasing on $[\frac{\pi}{2},\pi]$. In addition, we have $0\leq \gamma_t(\theta) \leq 1$.
\end{proposition}
\subsection{Some operator inequalities}
In the first result, we improve the mixed Cauchy-Schwarz inequality in reproducing kernel Hilbert space $\mathcal{H}(\Omega)$, then (see Kato\cite{1})
$$|\langle T\hat{k}_\lambda, \hat{k}_\mu\rangle|\leq \sqrt{\langle |T|^{2\nu}\hat{k}_\lambda, \hat{k}_\lambda\rangle\langle |T|^{2(1-\nu)}\hat{k}_\mu, \hat{k}_\mu\rangle},$$
$\lambda,\mu \in \Omega$, $0\leq \nu \leq 1$. This inequality has been used extensively in the literature when dealing with numerical radius and Berezin radius inequalities, see, for instance, \cite{15,9,6,13,12,16}.
\begin{theorem}\label{5.8}
Let $T\in B(\mathcal{H}(\Omega))$ with the polar decomposition $T=U|T|$ and let $\lambda,\mu \in \Omega$. Then for any $\nu\in[0,1]$,
$$\|T\|_{B,2}\leq \mu(\theta)\sqrt{\text{ber} (|T|^{2\nu})\text{ber} \left(|T^*|^{2(1-\nu)}\right)},$$
where $\theta = \angle_{|T|^{\nu}\hat{k}_\lambda,|T|^{1-\nu}U^*\hat{k}_\mu}$.
\end{theorem}
\begin{proof}
According to the assumptions and by using Proposition we have
\begin{equation*}
\begin{split}
|\langle T\hat{k}_\lambda, \hat{k}_\mu\rangle|&= |\langle U|T|\hat{k}_\lambda, \hat{k}_\mu\rangle|\\
&= |\langle U|T|^{1-\nu}|T|^\nu\hat{k}_\lambda, \hat{k}_\mu\rangle|\\
&= |\langle |T|^\nu\hat{k}_\lambda, |T|^{1-\nu}U^*\hat{k}_\mu\rangle|\\
&\leq \mu(\theta)\||T|^\nu\hat{k}_\lambda\|\||T|^{1-\nu}U^*\hat{k}_\mu\|\\
&=\mu(\theta)\sqrt{\langle |T|^{2\nu}\hat{k}_\lambda,\hat{k}_\lambda\rangle \langle U|T|^{2(1-\nu)}U^* \hat{k}_\mu,\hat{k}_\mu\rangle}\\
&=\mu(\theta)\sqrt{\langle |T^{2\nu}|\hat{k}_\lambda,\hat{k}_\lambda\rangle \langle |T^*|^{2(1-\nu)} \hat{k}_\mu,\hat{k}_\mu\rangle},
\end{split}
\end{equation*}
as desired.
\end{proof}
Now we are ready to present a new bound for the Berezin radius.
\begin{corollary}
Let $T\in B(\mathcal{H}(\Omega))$ have the polar decomposition $T=U|T|$, $0\leq \nu \leq 1$ and let $\theta_\lambda =  \angle_{|T|^{\nu}\hat{k}_\lambda,|T|^{1-\nu}U^*\hat{k}_\lambda}$, where $\lambda\in\Omega$.
\begin{enumerate}
\item[(i)] If $0\leq \theta_1< \theta_\lambda <\theta_2\leq \frac{\pi}{2}$ for all $\lambda\in\Omega$, then
$$\text{ber}(T) \leq \frac{\mu(\theta_1)}{2}\left\||T|^{2\nu}+|T^*|^{2(1-\nu)}\right\|_{B,1}.$$
\item[(2)] If $\frac{\pi}{2}\leq \theta_1< \theta_\lambda <\theta_2\leq \pi$ for all $\lambda\in\Omega$, then
$$\text{ber}(T) \leq \frac{\mu(\theta_2)}{2}\left\||T|^{2\nu}+|T^*|^{2(1-\nu)}\right\|_{B,1}.$$
\end{enumerate}
\end{corollary}
\begin{proof}
$(i)$ Let $\lambda\in\Omega$ be arbitrary. By Theorem \ref{5.8}, we have
\begin{equation*}
\begin{split}
|\langle T\hat{k}_\lambda, \hat{k}_\mu\rangle|&\leq \mu(\theta_\lambda)\sqrt{\langle |T^{2\nu}|\hat{k}_\lambda,\hat{k}_\lambda\rangle \langle |T^*|^{2(1-\nu)} \hat{k}_\lambda,\hat{k}_\lambda\rangle}\\
&\leq \mu(\theta_1 \sqrt{\langle |T^{2\nu}|\hat{k}_\lambda,\hat{k}_\lambda\rangle \langle |T^*|^{2(1-\nu)} \hat{k}_\lambda,\hat{k}_\lambda\rangle}\\
&\leq \mu(\theta_1)\left(\frac{\langle |T^{2\nu}|\hat{k}_\lambda,\hat{k}_\lambda\rangle+ \langle |T^*|^{2(1-\nu)} \hat{k}_\lambda,\hat{k}_\lambda\rangle}{2}\right)\\
&= \frac{\mu(\theta_1)}{2}\langle (|T^{2\nu}+ |T^*|^{2(1-\nu)} )\hat{k}_\lambda,\hat{k}_\lambda\rangle\\
\leq \frac{\mu(\theta_2)}{2}\left\||T|^{2\nu}+|T^*|^{2(1-\nu)}\right\|_{B,1},
\end{split}
\end{equation*}
where the second inequality is obtained from Corollary \ref{5.5} and the third inequality follows from the arithmetic-geometric mean inequality. Therefore
$$|\langle T\hat{k}_\lambda, \hat{k}_\mu\rangle| \leq \frac{\mu(\theta_2)}{2}\left\||T|^{2\nu}+|T^*|^{2(1-\nu)}\right\|_{B,1},$$
this implies the desired result by taking the supremum over all $\lambda\in\Omega$.
\end{proof}
\begin{corollary}
Let $T\in B(\mathcal{H}(\Omega))$ have the polar decomposition $T=U|T|$, $0\leq \nu \leq 1$ and let $\theta_\lambda =  \angle_{|T|^{\nu}\hat{k}_\lambda,|T|^{1-\nu}U^*\hat{k}_\lambda}$, where $\lambda\in\Omega$.
\begin{enumerate}
\item[(i)] If $0\leq \theta_1< \theta_\lambda <\theta_2\leq \frac{\pi}{2}$, then
$$\cos(\theta_2) \left\||T|^{2\nu}\#|T^*|^{2(1-\nu)}\right\|\leq \text{ber}(T) .$$
\item[(2)] If $\frac{\pi}{2}\leq \theta_1< \theta_\lambda <\theta_2\leq \pi$, then
$$\cos(\theta_1) \left\||T|^{2\nu}\#|T^*|^{2(1-\nu)}\right\|\leq \text{ber}(T) .$$
\end{enumerate}
\end{corollary}
\begin{proof}
We prove only case $(i)$ since the case $(ii)$ can be proved similarly. So, in this case, by Proposition \ref{5.7}, we have $\gamma_t(\theta_2)\leq \gamma_t(\theta)$. Putting $x = |T|^{\nu}\hat{k}_\lambda,$ $y = |T|^{1-\nu}U^*\hat{k}_\lambda,$ in inequality \ref{eq5.2}, we obtain
\begin{equation*}
\begin{split}
\cos &(\theta_2)\langle |T|^{2\nu}\# |T^*|^{2(1-\nu)}\hat{k}_\lambda,\hat{k}_\lambda\rangle\\ &\leq \cos(\theta_2)\left[ \langle |T|^{2\nu}\hat{k}_\lambda,\hat{k}_\lambda\rangle \langle  |T^*|^{2(1-\nu)}\hat{k}_\lambda,\hat{k}_\lambda\rangle\right]^{\frac{1}{2}}\\
&= \cos(\theta_2)\left[ \langle |T|^{2\nu}\hat{k}_\lambda,\hat{k}_\lambda\rangle \langle U |T|^{2(1-\nu)}U^*\hat{k}_\lambda,\hat{k}_\lambda\rangle\right]^{\frac{1}{2}}\\
&\leq |\langle |T|^\nu \hat{k}_\lambda, |T|^{1-\nu}U^*\hat{k}_\lambda|\\
&= |\langle T\hat{k}_\lambda, \hat{k}_\mu\rangle|.
\end{split}
\end{equation*}
Note that the first inequality follows from the fact for strictly positive operators $T,S$,
$$\langle T\# S \hat{k}_\lambda,\hat{k}_\lambda\rangle \leq \sqrt{\langle T \hat{k}_\lambda,\hat{k}_\lambda\rangle\langle  S \hat{k}_\lambda,\hat{k}_\lambda\rangle}.$$
Hence 
$$\cos (\theta_2)\langle |T|^{2\nu}\# |T^*|^{2(1-\nu)}\hat{k}_\lambda,\hat{k}_\lambda\rangle \leq |\langle T\hat{k}_\lambda, \hat{k}_\mu\rangle|,$$
which implies $\cos(\theta_2) \left\||T|^{2\nu}\#|T^*|^{2(1-\nu)}\right\|\leq \text{ber}(T),$ as desired.
\end{proof}
\section{On the convex hull of Berezin range}
Let $T$ be a bounded linear operator on the Hardy space. The spectrum of $T$ is denoted by $\sigma(T)$. We write $\text{conv} E$ to denote the convex hull of a set $E$. Recall that the numerical range of $T$, denoted $W(T)$, is the set
$$ W(T) = \{\langle Tu,u\rangle :u\in\mathcal{H}~~ \text{and}~~ \|u\|=1\}.$$
It is well known that $W(T)$ is always convex and its closure, $\overline{W(T)}$, contains the convex hull of spectrum. If $T$ is normal, then $\overline{W(T)} = \text{conv}\sigma(T)$. It is known that Berezin range is a subset of $W(T)$ and it is not necessarily convex at that. So, it is natural to ask: can we claim that the convex hull of $\text{Ber}(T)$ coincides with the closure of $W(T)$?

In this short section, we show that the answer to this question is, in general, negative.
\begin{example}
Let $S: H^2(\mathbb{D}) \to H^2(\mathbb{D})$ be the unilateral shift operator. We consider the following operator
$$N_z:= S(I-SS^*).$$
It is easy to prove that:
\begin{enumerate}
\item[1)]$\|N_z\| =1;$
\item[2)]$N_z^2 = 0;$
\item[3)]$N_z$ is a $1$-dimensional operator (compact);
\item[4)]$\text{Ber}(N_z)=\{z(1-|z|^2):z\in\mathbb{D}\};$
\item[5)]$W(N_z)=\overline{\mathbb{D}}_{\cos \frac{\pi}{2+1}}=\overline{\mathbb{D}}_{\frac{1}{2}}.$
\end{enumerate}
Since the function $f(x) := x(1-x^2),$ $ 0\leq x<1$, has supremum $\sup_{0<x<1}x(1-x^2) =\frac{2}{3\sqrt{3}}$, we have that $\text{ber}(N_z)=\frac{2}{3\sqrt{3}}$. So, since $\text{Ber}(N_z)\subset \mathbb{D}_{\frac{2}{3\sqrt{3}}}$ and $\frac{2}{3\sqrt{3}}<\frac{1}{2}(=w(N_z)),$ we deduce that $\text{convBer}(N_z)$ is a proper subset in $\overline{W(N_z)}$, i.e., $\text{convBer}(N_z)\neq \overline{W(N_z)}$.
\end{example} 

Now a natural question would be to describe class of operators $T$ for which $\text{convBer}(T)$ is equal (or not equal) to $\overline{W(T)}$?

These questions were  raised after the discussion of second author with Professor I.M.Spitkovsky.

Now we provide some positive answers to the above questions. Consider the reproducing kernel Hilbert space $\mathbb{C}^n$. The kernels in $\mathbb{C}^n$ \cite{paulsen2016introduction} is the standard orthonormal basis of $\mathbb{C}^n$, $e_i = (0,\cdots,1,\cdots, 0)$ where 1 occurs in the $i$th position. So for any $u \in \mathbb{C}^n$ with $\|u\|$=1, we have	$u = \lambda_{1}e_{1} + ..... + \lambda_{n}e_{n}$, where $|\lambda_{1}|^{2} + ..... + |\lambda_{n}|^{2} = 1$.
  
  \begin{theorem}\label{finite}
  	If $A$ is an $n\times n$ complex diagonal matrix, then its numerical range is the convex hull of its Berezin range, i.e.,
  		$$W(A) = \left\{ \displaystyle\sum_{i=1}^{n} |\lambda_{i}|^{2} \widetilde{A}(i) : \displaystyle\sum_{i=1}^{n} |\lambda_{i}|^{2} = 1 \right\}. $$
  \end{theorem}
\begin{proof}
For each $u$ with $\|u\| = 1,$
$$\langle Au,u\rangle = \displaystyle\sum_{i=j}|\lambda_{i}|^{2} \langle Ae_{i},e_{i}\rangle + \displaystyle\sum_{i \neq j}\lambda_{i}\bar{\lambda_{j}} \langle Ae_{i},e_{j}\rangle.$$
Since A is a diagonal matrix, we have $\langle Ae_{i},e_{j}\rangle = 0,~~\forall~~~ i \neq j.$
Therefore 
$$\langle Au,u\rangle = \displaystyle\sum_{i=j}|\lambda_{i}|^{2} \langle Ae_{i},e_{i}\rangle = \sum_{i=1}^{n} |\lambda_{i}|^{2} \widetilde{A}(i).$$
\end{proof}
\begin{corollary}
	If $A$ is an $n\times n$ complex diagonal matrix with constant diagonal c, then
$$W(A) = \text{Ber}(A) = \sigma(A) = \{c\}.$$
\end{corollary}
\begin{corollary}
 If an n$\times$n matrix A is unitary equivalent to a n$\times$n diagonal matrix D, then the numerical range of A is given by
$$W(A) = \left\{ \displaystyle\sum_{i=1}^{n} |\lambda_{i}|^{2} \widetilde{D}(i) : \displaystyle\sum_{i=1}^{n} |\lambda_{i}|^{2} = 1 \right\}. $$
\end{corollary}
\begin{proof}
Since the numerical range is invariant under unitary equivalence, this result follows by Theorem \ref{finite}.
\end{proof}

Similar to the case of $\mathbb{C}^n$, the kernels in $\ell^2$ are the standard orthonormal basis, $e_i = (0,\cdots,1,\cdots )$ where 1 occurs in the $i$th position. So for any $u \in \ell^2$ with $\|u\|$ = 1, we have $u = \displaystyle\sum_{i=1}^{\infty}\lambda_{i}e_{i}$ where $\displaystyle\sum_{i=1}^{\infty}|\lambda_{i}|^{2} = 1$.
\begin{theorem}\label{infty}
If $A$ is a complex diagonal matrix, then its numerical range is the convex hull of its Berezin range, i.e.,
$$W(A) = \left\{ \displaystyle\sum_{i=1}^{\infty} |\lambda_{i}|^{2} \widetilde{A}(i) : \displaystyle\sum_{i=1}^{\infty} |\lambda_{i}|^{2} = 1 \right\} $$.
\end{theorem}
\begin{proof}
For each $u$ with $\|u\| = 1,$
$$\langle Au,u\rangle = \displaystyle\sum_{i=j}|\lambda_{i}|^{2} \langle Ae_{i},e_{i}\rangle + \displaystyle\sum_{i \neq j}\lambda_{i}\bar{\lambda_{j}} \langle Ae_{i},e_{j}\rangle.$$
Since A is a diagonal matrix, we have $\langle Ae_{i},e_{j}\rangle = 0~~~\forall ~~~ i \neq j$. Therefore 
$$\langle Au,u\rangle = \displaystyle\sum_{i=j}|\lambda_{i}|^{2} \langle Ae_{i},e_{i}\rangle=\sum_{i=1}^{\infty} |\lambda_{i}|^{2} \widetilde{A}(i).$$
\end{proof}

\begin{corollary}
	If $A$ is an infinite dimensional complex diagonal matrix with constant diagonal c, then
$$W(A) = \text{Ber}(A) = \sigma(A) = \{c\}.$$
\end{corollary}
\begin{corollary}
 If an infinite dimensional matrix A is unitary equivalent to an infinite dimensional diagonal matrix D, then the numerical range of A is given by
$$W(A) = \left\{ \displaystyle\sum_{i=1}^{\infty} |\lambda_{i}|^{2} \widetilde{D}(i) : \displaystyle\sum_{i=1}^{\infty} |\lambda_{i}|^{2} = 1 \right\}. $$
\end{corollary}
\begin{proof}
Since the numerical range is invariant under unitary equivalence, this result follows by Theorem \ref{infty}.
\end{proof}

The following example shows that, in general, the equality need not hold even in the finite-dimensional matrix setting.
\begin{example}
Consider the canonical form of  $2\times 2$ matrix,
$$
	A=\begin{bmatrix} 
	\lambda_1 & m \\
	0 & \lambda_2\\
	\end{bmatrix}.
	$$
	If $\lambda_1 \neq \lambda_2$ and $m\neq 0$, by elliptic range theorem \cite{gustafson1997numerical} the numerical range of $T$, $W(T)$ is an ellipse with foci at $\lambda_1,\lambda_2$, minor axis $|m|$ and major axis $\sqrt{4r^2+|m|^2}$, where $\frac{\lambda_1 - \lambda_2}{2} = re^{i\mu}$.
	But $\text{Ber}(A) = \{\lambda_1,\lambda_2\}$. So $\text{conv} A$ is the line joining $\lambda_1$ and $\lambda_2$. Therefore $\text{conv} A$ is a proper subset of $W(A)$.
\end{example}

\textbf{Declaration of competing interest}

There is no competing interest.\\

\textbf{Author Contribution}

Athul Augustine, M. Garayev and P. Shankar contributed equally to the conceptualization, writing, review, and approval of the manuscript.\\

\textbf{Funding} 

No funding.\\

\textbf{Data availability}

No data was used for the research described in the article.\\

{\bf Acknowledgments.} The authors are grateful to the referees for the useful comments and suggestions. The first author is supported by the Senior Research Fellowship (09/0239(13298)/2022-EMR-I) of CSIR (Council of Scientific and Industrial Research, India). The second author is supported by the Researchers Supporting Project number(RSPD2024R1056), King Saud University, Riyadh, Saudi Arabia. The third author is supported by the Teachers Association for Research Excellence (TAR/2022/000063) of SERB (Science and Engineering Research Board, India).

\nocite{*}
\bibliographystyle{amsplain}
\bibliography{database}  
\end{document}